\documentclass[a4paper,12pt]{article}
\usepackage{amsfonts,amsthm,amsmath}

\usepackage[mathscr]{eucal}
\usepackage{amsmath}
\usepackage{amsfonts}
\usepackage{amssymb}
\usepackage{footnpag}
\usepackage[dvips]{graphicx}

\usepackage[mathlines]{lineno}

\theoremstyle{plain}
\newtheorem{thm}{Theorem}[section]
\newtheorem{prop}{Proposition}[section]
\newtheorem{lem}{Lemma}[section]
\newtheorem{cor}{Corollary}[section]

\theoremstyle{definition}

\newtheorem{rem}{Remark}[section]

\newcommand{\Q}{\mathbb{Q}}

\newcommand{\R}{\mathbb{R}}

\newcommand{\OOO}{\mathcal{O}}

\newcommand{\oo}{\mathfrak{o}}

\newcommand{\QQ}{\mathbb{Q}}

\newcommand{\NN}{\mathbb{N}}
\newcommand{\ZZ}{\mathbb{Z}}
\newcommand{\RR}{\mathbb{R}}
\newcommand{\CC}{\mathbb{C}}

\newcommand{\Cl}{\mathop{\mathrm{Cl}}\nolimits}

\newcommand{\e}{\hfill $\Box$}

\begin{document}
\title{An elementary approach to toy models for \\D. H. Lehmer's conjecture}
\author{
Eiichi Bannai\thanks{Graduate School of Mathematics Kyushu University, 
Motooka 744 Nishi-ku, Fukuoka, 819--0395 Japan. 
email: bannai@math.kyushu-u.ac.jp}, 
Tsuyoshi Miezaki\thanks{Department of Mathematics, Hokkaido University, 
Kita 10 Nishi 8 Kita-Ku, Sapporo, Hokkaido, 060--0810 Japan, 
e-mail: miezaki@math.sci.hokudai.ac.jp} and 
Vladimir A. Yudin\thanks{Moscow Power Engineering Institute
(Technical University), 
I05835 Moscow ,Russia, e-mail: vlayudin@mtu-net.ru. 
}
}

\maketitle



\begin{quote}
{\small\bfseries Abstract.}
In 1947, Lehmer conjectured that the Ramanujan's tau function $\tau (m)$ 
never vanishes for all positive integers $m$, 
where $\tau (m)$ is the $m$-th Fourier coefficient of the cusp form 
$\Delta _{24}$ of weight $12$. 
The theory of spherical $t$-design is closely related to 
Lehmer's conjecture because 
it is shown, by Venkov, de la Harpe, and Pache, that 
$\tau (m)=0$ is equivalent to the fact that 
the shell of norm $2m$ of the $E_{8}$-lattice is a spherical $8$-design. 
So, Lehmer's conjecture is reformulated in terms of spherical $t$-design. 

Lehmer's conjecture is difficult to prove, and still remains open. 
However, Bannai-Miezaki showed that 
none of the nonempty shells of the integer 
lattice $\ZZ^2$ in $\RR^2$ is a spherical $4$-design, 
and that none of the nonempty shells of the hexagonal lattice 
$A_2$ is a spherical $6$-design. Moreover, 
none of the nonempty shells of the integer lattices associated to the 
algebraic integers of 
imaginary quadratic fields whose class number is either $1$ or $2$, 
except for $\QQ(\sqrt{-1})$ and $\QQ(\sqrt{-3})$ 
is a spherical $2$-design. 
In the proof, the theory of modular forms played an important role. 

Recently, Yudin found an elementary proof for the case of $\ZZ^{2}$-lattice 
which does not use the theory of modular forms but 
uses the recent results of Calcut. 
In this paper, we give the elementary (i.e., modular form free) proof 
and discuss the relation 
between Calcut's results and the theory of imaginary quadratic fields. 

\vspace{-0.15in}

\quad
\end{quote}

\section{Introduction}
It was shown by Bannai-Miezaki \cite{Toy-BM} that 
none of the nonempty shells of the integer 
lattice $\ZZ^2$ in $\RR^2$ is a spherical $4$-design, 
and that none of the nonempty shells of the hexagonal lattice 
$A_2$ is a spherical $6$-design. 
We called these results as toy models for D. H. Lehmer's conjecture, 
because the original Lehmer's conjecture 
that the value of the Ramanujan's tau function $\tau(m)$ 
is never zero for any positive integer $m$ is equivalent to the 
statement that no shell of the $E_8$-lattice 
(Korkine-Zolotareff lattice) is a spherical $8$-design, 
as it was observed by Venkov, de la Harpe, and Pache 
(cf.~{\cite{{HP}, {HPV}, {Pache}, {Venkov1}}). 
In \cite{Toy-BM} and in the subsequent \cite{Toy-BM2}, 
where further toy models (of Lehmer's conjecture) 
were obtained for the lattices associated to the algebraic 
integers of imaginary quadratic number fields whose class number 
is either $1$ or $2$, as well as in the work on Venkov, 
de la Harpe and Pache, 
the theory of modular forms played an important role. 
The third author (Yudin) seeked and then found an elementary 
proof (for the case of $\ZZ^2$-lattice) 
which does not use the theory of 
modular forms, just by using the language of Gaussian integers $\ZZ[\sqrt{-1}]$. 
In that proof, the recent results of Calcut \cite{Calcut} 
for Gaussian integers $\ZZ[\sqrt{-1}]$ was used, 
instead of modular forms, in some crucial ways. 
This paper describes an elementary approach, and 
the subsequent discussions among the three authors on this and related topics.
The main points of this paper are as follows. 
\begin{enumerate}
\item 
[(i)]
First we give an elementary proof for the $\ZZ^2$-lattice 
using the results of Calcut. 
\item 
[(ii)] 
We remark that the results of Calcut is essentially equivalent to 
the multiplicative property 
of the numbers of nonequivalent integral ideals of a 
certain imaginary quadratic number field, 
which is well known and is also directly proved 
in an elementary way. 
\item 
[(iii)] 
By using these elementary (i.e., modular form free) approach, 
we give an alternative proof for the lattice associated to 
the algebraic integers of 
any imaginary quadratic number field of class number is $1$. 
(Here we remark that we can also avoid the use of the results 
of Calcut \cite{Calcut}.) 
\item 
[(iv)] 
We formulate and prove generalizations of the results of 
Calcut \cite{Calcut} for $\ZZ[\sqrt{-1}]$ to 
the imaginary quadratic number fields whose class number is $1$. 
\end{enumerate} 
So, we are able to obtain the toy models for 
the lattice of the algebraic integers of 
imaginary quadratic fields with class number $1$ 
by an elementary approach, 
in a sense that it is modular form free and 
also free from the results of Calcut \cite{Calcut}. 

This paper is organized as follows. 
In Section \ref{section:pre}, we review 
the concept of spherical designs and 
the theory of imaginary quadratic fields and 
quote the results of Calcut \cite{Calcut}. 
In Section \ref{section:non}, we study the 
nonexistence of the spherical designs in the 
shells of lattices. 
In Section \ref{section:non_Z^2_1} and \ref{section:non_Z^2_2}, 
we show that 
\begin{thm}\label{thm:Z^2}
The shells in $\ZZ^{2}$-lattice are not spherical $4$-designs. 
\end{thm}
In Section \ref{subsection:mul}, 
we show that the results of Calcut is essentially equivalent to 
the multiplicative property 
of the numbers of nonequivalent integral ideals of a 
certain imaginary quadratic number field. 
In Section \ref{section:non_A_2}, we show that 
\begin{thm}\label{thm:A_2}
The shells in $A_2$-lattice are not spherical $6$-designs. 
\end{thm}
In Section \ref{section:non_gen}, we show that 
\begin{thm}\label{thm:gen}
Let $K=\Q(\sqrt{-d})$ be an imaginary quadratic field 
whose class number is $1$ and $d\neq 1$, $3$ 
i.e., $d$ is one of the following numbers 
$2$, $7$, $11$, $19$, $43$, $67$, $163$. 
Then, the shells in the lattice associated to $K$ 
are not spherical $2$-designs. 
\end{thm}
In Section \ref{subsection:calcut}, 
we study the generalization of Calcut's results. 


\section{Preliminary}\label{section:pre}
\subsection{Spherical designs}
The concept of a spherical $t$-design is due to Delsarte-Goethals-Seidel \cite{DGS}. For a positive integer $t$, a finite nonempty set $X$ on the unit sphere
\[
S^{n-1} = \{x = (x_1, x_2, \cdots , x_n) \in \R ^{n}\ |\ x_1^{2}+ x_2^2+ \cdots + x_n^{2} = 1\}
\]
is called a spherical $t$-design in $S^{n-1}$ if the following condition is satisfied:
\[
\frac{1}{|X|}\sum_{x\in X}f(x)=\frac{1}{|S^{n-1}|}\int_{S^{n-1}}f(x)d\sigma (x), 
\]
for all polynomials $f(x) = f(x_1, x_2, \cdots ,x_n)$ of degree not exceeding $t$. Here, the righthand side means the surface integral on the sphere, and $|S^{n-1}|$ denotes the surface volume of the sphere $S^{n-1}$. The meaning of spherical $t$-design is that the average value of the integral of any polynomial of degree up to $t$ on the sphere is replaced by the average value at a finite set on the sphere. A finite subset $X$ in $S^{n-1}(r)$, the sphere of radius $r$, 
is also called a spherical $t$-design if $\frac{1}{r}X$ is 
a spherical $t$-design on the unit sphere $S^{n-1}$.

We denote by ${\rm {\rm Harm}}_{j}(\R^{n})$ the set of homogeneous harmonic polynomials of degree $j$ on $\R^{n}$. It is well known that $X$ is a spherical $t$-design if and only if the condition 
\begin{eqnarray*}
\sum_{x\in X}P(x)=0 
\end{eqnarray*}
holds for all $P\in {\rm Harm}_{j}(\R^{n})$ with $1\leq j\leq t$. 
Moreover, if $X\subset S^{1}(r)$ then, 
since $\mbox{Harm}_{k}(S^{1})=\langle \mbox{Re}(z^k), 
\mbox{Im}(z^k)\rangle$, 
the following proposition holds:
\begin{prop}\label{prop:R^2}
Let $X=\{\xi_1,\ldots,\xi_{n}\}\subset S^{1}$. 
We regard $S^1$ as complex numbers whose absolute values are one, 
namely, $S^1\simeq \{z\in\CC\mid\vert z\vert=1\}$. 
Then, $X$ is a spherical $t$-design if and only if 
\begin{eqnarray*}
\sum_{i=1}^{n}\xi_{i}^{k}=0, 
\end{eqnarray*}
for all $k\in \{1,\ldots,t\}$. 
\end{prop}
For a lattice $\Lambda$ and a positive real number $m>0$, the shell of norm $m$ of $\Lambda$ is defined by 
\[
\Lambda_{m}:=\{x\in \Lambda\ |\ (x,x)=m \}=\Lambda\cap S^{n-1}(m), 
\]
where $(x,y)$ is the standard Euclidean inner product. 
The theta series $\Theta_{\Lambda}(q)$ of $\Lambda$ is the 
following formal power series
\[
\Theta_{\Lambda}(q)=
\sum_{x\in \Lambda}q^{(x, x)}
=\sum_{m=0}^{\scriptstyle \infty}\vert \Lambda_{m}\vert q^{m}.
\]
For example, when $\Lambda$ is the $\ZZ^{2}$-lattice
\begin{eqnarray*}
\Theta_{\Lambda}(q)=\theta_3(q)^2
&=&1+\sum_{m=1}^{\infty}r_2(m) q^{m} \\
&=&1 + 4q + 4q^2 + 4q^4 + 8q^5 + 4q^8 +\cdots, 
\end{eqnarray*}
where $\theta_3(q)=1+\sum_{i=1}^{\infty}2q^{i^2}$ and 
the coefficient $r_{2}(m)$ is a number of ways of writing $m$ 
as a sum of $2$ squares. 

\subsection{Imaginary quadratic fields}
In this subsection, we review the theory of an imaginary quadratic field. 
Let $K=\Q(\sqrt{-d})$ be an imaginary quadratic field, 
and let $\mathcal{O}_{K}$ be its ring of algebraic integers. 
Let $\Cl_{K}$ be the ideal classes. In this paper, 
we only consider the case $\vert \Cl_{K}\vert =1$. 
So, we denote by 
$\oo$ the principal ideal class. 
We denote by $d_{K}$ the discriminant of $K$:
\begin{eqnarray*}
d_{K}=\left\{
\begin{array}{lll}
-4d\ &{\rm if }\ -d\equiv 2,\ 3&\pmod {4}, \\
-d\ &{\rm if }\ -d\equiv 1 &\pmod {4}. 
\end{array}
\right.
\end{eqnarray*}
\begin{thm}[cf.~{\cite[page~87]{Zagier}}]\label{thm:lattice}
Let $d$ be a positive square-free integer, and let $K=\Q(\sqrt{-d})$. 
Then 
\begin{eqnarray*}
\mathcal{O}_{K}=\ZZ+\ZZ\,\theta_d, 
\end{eqnarray*}
where 
\begin{eqnarray}\label{df:theta_d}
\theta_d=
\left\{
\begin{array}{lll}
\sqrt{-d}\quad &if\ -d\equiv 2,\ 3 &\pmod {4}, \\
\displaystyle\frac{1+\sqrt{-d}}{2}\quad &if\ -d\equiv 1 &\pmod {4}. 
\end{array}
\right.
\end{eqnarray}
\end{thm}
Therefore, we consider $\OOO_{K}$ to be the lattice in $\R^{2}$ 
with the basis 
\begin{eqnarray*}
\left\{
\begin{array}{lll}
\displaystyle(1, 0), (1,\sqrt{-d} )\quad &\mbox{if}\ -d\equiv 2,\ 3 &\pmod {4}, \\
\displaystyle(1, 0), \Big(\frac{1}{2},\frac{\sqrt{-d}}{2}\Big)\quad &
\mbox{if}\ -d\equiv 1 &\pmod {4}, 
\end{array}
\right.
\end{eqnarray*}
denoted by $L_{\oo}$. 

It is well-known that there exists a one-to-one correspondence 
between the set of reduced quadratic forms $f(x, y)$ 
with a fundamental discriminant $d_K<0$ 
and the set of fractional ideal classes of the unique quadratic field 
$\QQ(\sqrt{-d})$ \cite[page~94]{Zagier}. Namely, 
For a fractional ideal $\mathfrak{a}=\ZZ\alpha +\ZZ\beta$, 
we obtain the quadratic form $ax^2+bxy+cy^2$, where 
$a=\alpha \overline{\alpha}/N(\mathfrak{a})$, 
$b=(\alpha\overline{\beta}+ \overline{\alpha}\beta)/N(\mathfrak{a})$ and 
$c=\beta \overline{\beta}/N(\mathfrak{a})$. 
Conversely, for a quadratic form $ax^2+bxy+cy^2$, 
we obtain the fractional ideal $\ZZ +\ZZ(b+\sqrt{d_K})/2a$. 
We remark that $N(\mathfrak{a})$ is the norm of $\mathfrak{a}$ and 
$\overline{\alpha}$ is a complex conjugate of $\alpha$. 
For example, $\ZZ+\ZZ\sqrt{-1}$, 
which is the principal ideal of $\QQ(\sqrt{-1})$, 
corresponds to $x^2+y^2$, that is, 
the $\ZZ^{2}$-lattice. 

Here, we define the automorphism group of $f(x, y)$ as follows: 
\begin{eqnarray*}
U_{f}=\left\{
\begin{pmatrix}
\alpha &\beta \\
\gamma &\delta
\end{pmatrix}
\in SL_{2}(\ZZ) \Biggm\vert f(\alpha x+\beta y, \gamma x+\delta y)=f(x, y)
\right\}
\end{eqnarray*}
Then, for $n\geq 1$, the number of the nonequivalent solutions of $f(x, y)=n$ 
under the action of $U_{f}$ 
is equal to the number of the integral ideal of norm $n$ \cite{Zagier}. 
\begin{thm}[cf.~{\cite[page~63]{Zagier}}]\label{thm:Zagier}
Let $f(x, y)$ be the reduced quadratic form 
with a fundamental discriminant $D<0$ and 
$U_{f}$ be the automorphism group of $f(x, y)$. 
Then 
\begin{eqnarray*}
\sharp U_{f}=
\left\{
\begin{array}{ll}
6 &{\text if}\ D=-3, \\
4 &{\text if}\ D=-4, \\
2 &{\text if}\ D<-4. 
\end{array}
\right.
\end{eqnarray*}
\end{thm}
These classical results are due to Gauss, Dirichlet, etc. 
Let $\mathfrak{a}$ be an ideal class and 
$f_{\mathfrak{a}}(x, y)$ be the reduced quadratic form 
corresponding to $\mathfrak{a}$. 
Moreover, let $L_{\mathfrak{a}}$ be the lattice corresponding to $f(x, y)$. 
We denote by $N(A)$ the norm of an ideal $A$. 
Then, using Theorem \ref{thm:Zagier}, we have 
\begin{eqnarray*}
\sum_{x\in L_{\mathfrak{a}}}q^{(x, x)}=
1+\sharp U_{f}\sum_{n=1}^{\infty}
\sharp\{A \mid A\mbox{ is an integral ideal of }\mathfrak{a},\, 
N(A)=n\}\,q^{m}. 
\end{eqnarray*}
Moreover, let $\{\mathfrak{a_i}\}_{i=1}^{s}$ be the 
complete set of ideal classes of an imaginary quadratic field 
whose class number is $s$ and 
let $\{L_{\mathfrak{a}_i}\}_{i=1}^{s}$ be the lattices 
corresponding to $\{\mathfrak{a_i}\}_{i=1}^{s}$. 
We denote by $a(m)$ the $m$-th coefficient of 
the sum of theta functions:
\begin{eqnarray*}
\sum_{x\in L_{\mathfrak{a}_1}}q^{(x,x)}+\cdots 
+\sum_{x\in L_{\mathfrak{a}_s}}q^{(x,x)}=\sum_{m=0}^{\infty}a(m)q^{m}. 
\end{eqnarray*}
Then, since the prime ideal factorization is unique, 
the following proposition holds:
\begin{prop}[cf.~{\cite[page 101]{Zagier}}]\label{thm:mul}
$a^{\prime}(m):=a(m)/\sharp U_{f_{\mathfrak{a}_i}}$ have the multiplicative property. 
Namely, $a^{\prime}(mn)=a^{\prime}(m)a^{\prime}(n)$ if $(m,n)=1$. 
\end{prop}
For example, let $\mathfrak{o}=\ZZ[\sqrt{-1}]$ be the 
only ideal class of $\QQ(\sqrt{-1})$. 
Then, $L_{\mathfrak{o}}$ is the $\ZZ^{2}$-lattice and 
\begin{eqnarray*}
\Theta_{\ZZ^{2}}(q)=\sum_{x\in \ZZ^{2}}q^{(x,x)}=\theta_{3}^{2}(q)
=\sum_{m=0}^{\infty}a(m)q^{m}, 
\end{eqnarray*}
where $\theta_{3}(q)=1+\sum_{i=1}^{\infty}2q^{i^2}$. 
Therefore, the coefficients $a(m)/4$ have the multiplicative property. 

Finally, we give the classical theorems needed later. 
\begin{thm}[cf.~{\cite[page 104,~Proposition 5.16]{Cox}}]\label{thm:facprime}
We can classify the prime ideals of a quadratic field as follows{\rm :}
\begin{enumerate}
\item 
If $p$ is an odd prime and $(d_{K}/p)=1$ 
$($resp.\ $d_{K}\equiv 1 \pmod{8}$$)$ then 
$(p)=P \overline{P}\ (resp.\ (2)=P \overline{P})$, 
where $P$ and $\overline{P}$ are prime ideals 
with $P\neq \overline{P}$, 
$N(P)=N(\overline{P})=p$ $($resp.\ $N(P)=2$$)$. 
\item 
If $p$ is an odd prime and $(d_{K}/p)=-1$ 
$($resp.\ $d_{K}\equiv 5 \pmod{8}$$)$ then 
$(p)=P\ (resp.\ (2)=P)$, 
where $P$ is a prime ideal with $N(P)=p^{2}$ $($resp.\ $N(P)=4$$)$. 
\item 
If $p\ \vert\ d_{k}$ then 
$(p)=P^{2}$, 
where $P$ is a prime ideal with $N(P)=p$. 
\end{enumerate}

\end{thm}

\begin{prop}\label{prop:NUM}
Let $F(m)$ be the number of the integral ideals of norm $m$ of $K$. 
Let $p$ be a prime number. 
Then, if $p\neq 2$ 
\begin{eqnarray*}
F(p^{e})=\left\{
\begin{array}{lll}
e+1 &{ if}\ \left(d_{K}/p\right)=1, \\
(1+(-1)^e)/2 &{ if}\ \left(d_{K}/p\right)=-1, \\
1 &{ if}\ p\ \vert\ d_{K}, 
\end{array} 
\right.
\end{eqnarray*}
if $p= 2$ 
\begin{eqnarray*}
F(2^{e})=\left\{
\begin{array}{lll}
e+1 &{ if}\ d_{K}\equiv 1 \pmod{8}, \\
(1+(-1)^e)/2 &{ if}\ d_{K}\equiv 5 \pmod{8}, \\
1 &{ if}\ 2\ \vert\ d_{K}. 
\end{array} 
\right.
\end{eqnarray*}
\end{prop}
\begin{proof}
When $\left(d_{K}/p\right)=1$ i.e., $(p)=P \overline{P}$ and 
$P \neq \overline{P}$, 
since $P$ and $\overline{P}$ are only integral ideals 
of norm $p$, 
we have $F(p)=2$. 
Moreover, the integral ideals of norm $p^{e}$ are as follows: 
$P^{e}$, $P^{e-1} \overline{P}$, \ldots , 
$(\overline{P})^{e}$. So, we have $F(p^{e})=e+1$. 
The other cases can be proved similarly. 
\end{proof}
\subsection{The results of Calcut}\label{subsection:Calcut}
We collect Calcut's results needed later. 
\begin{lem}[cf.~\cite{Calcut}]\label{lem:Calcut_Z^2}
Let $z\neq 0$ be a Gaussian integer. There is 
a natural number $n$ such that $z^n$ is real 
if and only if $\arg z$ is a multiple of $\pi/4$. 
\end{lem}
\begin{cor}[cf.~\cite{Calcut}]\label{cor:Calcut_Z^2_1}
The only rational values of $\tan (k\pi/n)$ 
are $0$ and $\pm 1$. 
\end{cor}
\begin{cor}[cf.~\cite{Calcut}]\label{cor:Calcut_Z^2_2}
Let $z_i=a_i+b_i\sqrt{-1} \in \ZZ[\sqrt{-1}]$. 
If 
\[
\displaystyle\frac{k\pi}{n}
=\sum_{j=1}^{l}m_j\arctan\frac{b_j}{a_j}, 
\]
holds, where all variables are rational integers, then 
$k\pi/n = s\pi/4$ for some integer $s$. 
\end{cor}
In \cite{Calcut}, Calcut showed that ``Lemma \ref{lem:Calcut_Z^2} 
$\Rightarrow$ Corollary \ref{cor:Calcut_Z^2_1} and \ref{cor:Calcut_Z^2_2}". 
Here, we show that ``Corollary \ref{cor:Calcut_Z^2_1} $\Rightarrow$ 
Lemma \ref{lem:Calcut_Z^2}" and ``Corollary \ref{cor:Calcut_Z^2_2} 
$\Rightarrow$ Lemma \ref{lem:Calcut_Z^2}". 
Therefore, these three statements are equivalent to one another. \\
{\it Proof of ``Corollary \ref{cor:Calcut_Z^2_1} $\Rightarrow$ 
Lemma \ref{lem:Calcut_Z^2}}". 
Let $z=a+b\sqrt{-1}$. If $z^{n}$ is real then 
\[
(a+b\sqrt{-1})^n\in\RR \Rightarrow n\arg (a+b\sqrt{-1})=k\pi 
\Rightarrow \arg (a+b\sqrt{-1})=\frac{k\pi}{n}
\]
Because of Corollary \ref{cor:Calcut_Z^2_1}, 
the rational values of $\tan k\pi/n$ are $0$ and $\pm 1$. 
Therefore, $\arg z$ is a multiple of $\pi/4$.
If $\arg z$ is a multiple of $\pi/4$ then 
$z^4\in\RR$. 
This complete the proof of Lemma \ref{lem:Calcut_Z^2}. 
\e \\ 
{\it Proof of ``Corollary \ref{cor:Calcut_Z^2_2} $\Rightarrow$ 
Lemma \ref{lem:Calcut_Z^2}}". 
Corollary \ref{cor:Calcut_Z^2_1} is the special case of 
Corollary \ref{cor:Calcut_Z^2_2}. 
\e 
\section{Nonexistence of the spherical designs}\label{section:non}
In this section, we study the nonexistence of the spherical designs. 
First, we introduce some notation. 
Let $K$ be an imaginary quadratic field and 
$L$ be a lattice corresponding to $\OOO_K$. 
Then, for 
$m=p_1^{a_1} \cdots p_s^{a_s}
q_{s+1}^{a_{s+1}}\cdots q_u^{a_u}
r_{u+1}^{a_{u+1}}\cdots r_v^{a_v}$, 
where $(d_K/p_i)=1$, $(d_K/q_i)=-1$ and $r_i\vert d_K$ 
we define sets as follows: 
\begin{equation}\label{eqns:sets}
\left\{
\begin{array}{ll}
X(L_m)&:=
\{(a+b\theta_d)/\sqrt{m}\mid (a+b\theta_d)\in \OOO_K, N(a+b\theta_d)=m\}, 
\vspace{3pt}
\\ 
\vspace{3pt}
&\simeq 
\{x/\sqrt{m}\mid x\in L, (x,x)=m\} 
\vspace{3pt}
\\
\vspace{3pt}
X(L_m)_{p_k}&:=
\{(a+b\theta_d)/\sqrt{p_k^{a_k}}\mid (a+b\theta_d)\in \OOO_K, N(a+b\theta_d)=p_k^{a_k}\}, 
\vspace{3pt}
\\
\vspace{3pt}
&\simeq 
\{x/\sqrt{p_k^{a_k}}\mid x\in L, (x,x)=p_k^{a_k}\} 
\vspace{3pt}
\\
\vspace{3pt}
W(L_m)&:=
\{\arg (a+b\theta_d)\mid (a+b\theta_d)\in \OOO_K, N(a+b\theta_d)=m\}, 
\vspace{3pt}
\\
\vspace{3pt}
W(L_m)_{p_k}&:=
\{\arg (a+b\theta_d)\mid (a+b\theta_d)\in \OOO_K, N(a+b\theta_d)=p_k^{a_k}\}. 
\end{array}
\right.
\end{equation}
For a lattice $L$, we define functions as follows: 
\begin{equation}
\left\{
\begin{array}{ll}\label{func:sets}
I_{L_m}(t)&:=\displaystyle\frac{1}{|X(L_m)|}\sum_{x\in X(L_m)}x^t\\
I_{L_m,p_k}(t)&:=\displaystyle\frac{1}{|X(L_m)_{p_k}|}
\sum_{x\in X(L_m)_{p_k}}x^{t}. 
\end{array}
\right.
\end{equation}
Because of Proposition \ref{prop:R^2}, we remark that 
$L_m$ is a spherical $t$-design if and only if 
$I_{L_m}(k)=0$ for all $k\in\{1,\ldots,t\}$. 
Then, it is well-known that the following theorem: 
\begin{thm}[cf.~\cite{{Pache},{Toy-BM},{Toy-BM2}}]
Let $L$ be a $2$-dimensional Euclidean lattice. 
Then, for any positive integer $m$, 
$L_m$ is a spherical $1$-design if $L_m \neq \emptyset$. 
Moreover, $(\ZZ^{2})_m$ $($resp.\ $(A_2)_m$$)$ 
is a spherical $3$-design $($resp.\ $5$-design$)$ 
if $(\ZZ^{2})_m \neq \emptyset$ $($resp.\ $(A_2)_m \neq \emptyset$$)$. 
\end{thm}
\subsection{The case of $\ZZ^{2}$-lattice}\label{section:non_Z^2}
Let $\varphi_k\in (0, 2\pi)$ be the minimum argument 
$z\in\OOO_{\QQ(\sqrt{-1})}$ whose norm is $p_k$, where $(d_K/p_k)=1$. 
Then, we have the following lemma: 
\begin{lem}\label{lem:mul_prime_power}
Let the notation be the same as above. 
Then, for a prime number $p_k$, where $(d_K/p_k)=1$, 
\[
W(\ZZ^2_{m})_{p_k}=
\left\{
\begin{array}{l}
\displaystyle\{0,\pm 2\varphi_k, \pm 4\varphi_k,\ldots, \pm a_k\varphi_k\}\oplus \{\frac{\ell\pi}{2}\,(0\leq \ell <4)\}, 
\mbox{ if $a_k$ is even, }\\
\displaystyle\{\pm \varphi_k, \pm 3\varphi_k,\ldots, \pm a_k\varphi_k\}\oplus 
\{\frac{\ell\pi}{2}\,(0\leq \ell <4)\}, \mbox{ if $a_k$ is odd, }
\end{array}
\right.
\]
where denote by ``$\{a,\ldots\}\oplus 
\{\frac{\ell\pi}{2}\,(0\leq \ell <4)\}$" 
the set $\{a, a+\pi/2, a+\pi, a+3\pi/2,\ldots\}$ 
and $W(\ZZ^{2}_{m})_{p_k}$ is defined in {\rm (\ref{eqns:sets})}. 
In particular, $\vert W((\ZZ^{2})_{m})_{p_k}\vert=4(1+a_k)$. 
\end{lem}
\begin{proof}
If $a_k$ is even then the values of the arguments $W(\ZZ^{2}_{m})_{p_k}$ 
are $\pm \varphi_k\pm\cdots\pm \varphi_k\pmod{\pi/2}$, that is, 
one of the elements of the following set: 
$\{0,$$\pm 2\varphi_k$
$, \pm 4\varphi_k,\ldots$
$, \pm a_k\varphi_k\}$$ \pmod{\pi/2}$. 
Then, we assert that if $s\neq s^{\prime}$ then 
$s\varphi_k\not \equiv s^{\prime}\varphi_k\pmod{\pi/2}$. 
It is because if not then $\varphi_k=n\pi/(2t)$ for some $n$, $t\in\NN$ 
and $z^{2t}\in \RR$. 
However, because of Lemma \ref{lem:Calcut_Z^2}, 
$\varphi_k$ is a multiple of $\pi/4$. 
Then, 
if $\varphi_k=\pi/4$ then $z=c(1+\sqrt{-1})$ for some $c\in\RR$ and 
$p_k=2c^2$. 
This is a contradiction since $p_k$ is a prime number 
congruent to $1$ modulo $4$. 
If $\varphi_k=\pi/2$ then $z=c(\sqrt{-1})$ for some $c\in\RR$ and 
$p_k=c^2$. 
This is a contradiction since $p_k$ is a prime number. 
For the other cases, 
we can obtain a contradiction similarly. 
Therefore, we obtain $\vert W((\ZZ^{2})_{m})_{p_k}\vert=4(1+a_k)$. 
In case that $a_k$ is odd, it can be proved similarly. 
\end{proof}
\subsubsection{The proof using Calcut's result}\label{section:non_Z^2_1}
In this subsection, 
we prove Theorem \ref{thm:Z^2} using Calcut's results. 
For $m=p_1^{a_1} \cdots p_s^{a_s}q_{s+1}^{a_{s+1}}\cdots q_u^{a_u}$, 
where $p_i\equiv 1\pmod{4}$ and $q_i\equiv 3\pmod{4}$ 
we define the sets in (\ref{eqns:sets}). 
We denote by $(a+b\sqrt{-1})(a-b\sqrt{-1})$, 
where $b>0$, 
the prime ideal factorization of $(p_i)$. 
Then, we denote by $P_i$ (resp.\ $\overline{P_i}$) 
the prime ideal $(a+b\sqrt{-1})$ (resp. $(a-b\sqrt{-1})$). 
Here, we show that 
\begin{equation}\label{eqn:mul_Z^2_1}
\frac{\vert W((\ZZ^{2})_{m})\vert}{4}
=\frac{\vert W((\ZZ^{2})_{m})_{p_1}\vert}{4} \cdots \frac{\vert W((\ZZ^{2})_{m})_{p_s}\vert}{4}. 
\end{equation}
If not, $m_1\varphi_1+\cdots+m_s\varphi_s 
\equiv m^{\prime}_1\varphi_1+\cdots+m^{\prime}_s\varphi_s \pmod{\pi/2}$, 
namely, 
$(m_1-m^{\prime}_1)\varphi_1+\cdots+(m_s-m^{\prime}_s)\varphi_s 
\equiv 0 \pmod{\pi/2}$. 
Therefore, there exist $a^{\prime}_1,\ldots,a^{\prime}_s$ such that 
$(z):=P_1^{a^{\prime}_1}\cdots P_s^{a^{\prime}_s}$ and $z^2\in\RR$. 
Then, $\overline{P_i}$ is equal to $P_j$ for some $j$ 
because $z^2\in\RR$ and $p_i\equiv 1\pmod{4}$. 
This is a contradiction since $\{P_i\}_{i=1}^{s}$ are 
the prime ideals with the different norms. 

Then, we obtain the following equation: 
\begin{eqnarray}\label{eqn:design}
I_{(\ZZ^{2})_{m}}(t)=\frac{1}{|X((\ZZ^{2})_{m})|}\sum_{x\in X((\ZZ^{2})_{m})}x^t \nonumber
&=&\frac{1}{|X((\ZZ^{2})_{m})|}\sum_{x\in X((\ZZ^{2})_{m})}e^{it\arg (x)}\\
&=&\prod_{k=1}^{s}\frac{1}{|X((\ZZ^{2})_{m})_{p_k}|}\sum_{x\in X((\ZZ^{2})_{m})_{p_k}}e^{it\arg (x)}\nonumber\\
&=&\prod_{k=1}^{s}I_{(\ZZ^{2})_{m},p_k}(t). 
\end{eqnarray}
Let $\varphi_k\in (0, 2\pi)$ be the minimum argument 
$z\in\OOO_{\QQ(\sqrt{-1})}$ whose norm is $p_k$. 
Then, 
\begin{equation}\label{eqn:Lehmer_Z^2}
I_{(\ZZ^{2})_{m},p_k}(4)=\frac{\sin (4(1+a_k)\varphi_k)}{(1+a_k)\sin (4\varphi_k)} 
\end{equation}
since the following equations hold: 
\begin{equation}\label{eqn:cos}
\left\{
\begin{array}{l}
\vspace{5pt}
\displaystyle
1+2\cos(2x)+2\cos(4x)+\cdots +2\cos(2kx)=\frac{\sin((2k+1)x)}{\sin x}\\ 
\displaystyle
\cos x+\cos(3x)+\cos(5x)+\cdots +\cos((2k-1)x)=\frac{\sin(2kx)}{2\sin x}. 
\end{array}
\right.
\end{equation}
If $I_{(\ZZ^{2})_{m},p_k}(4)=0$, namely, 
$4\varphi_k = n\pi /(1+a_k)$ for some $n\in \ZZ$, 
then because of Corollary \ref{cor:Calcut_Z^2_1}, 
$\tan (4\varphi_k) = \pm 1$, namely, 
$\varphi_k =\pi /16$, $3\pi /16$, $5\pi /16$ and $7\pi /16$. 
Then, $\tan \varphi_k$ is an irrational number. 
On the other hand, 
$\varphi_k$ is the argument of 
$\ZZ[\sqrt{-1}]$, hence a rational number. 
This is a contradiction. 
Therefore $(\ZZ^{2})_m$ is not a spherical $4$-design. 

\subsubsection{The proof using multiplicative property}\label{section:non_Z^2_2}
In this subsection, 
we reprove Theorem \ref{thm:Z^2} using the multiplicative property 
Proposition \ref{thm:mul}. 
We denote by $a(m)$ the $m$-th coefficient of the theta series of 
the $\ZZ^{2}$-lattice:
\[
\Theta_{\ZZ^{2}}(q)=\sum_{i=0}^{\infty}a(i)q^i, 
\]
and set $a^{\prime}(m):=a(m)/4$. 
Then, 
because of Proposition \ref{thm:mul} and \ref{prop:NUM} 
the function $a^{\prime}(m)$ is multiplicative and 
when $m$ is a prime power, $a^{\prime}(p^e)=F(p^e)$. 
Therefore, 
for $m=p_1^{a_1} \cdots p_s^{a_s}q_{s+1}^{a_{s+1}}\cdots q_u^{a_u}2^{c}$, 
where $p_i\equiv 1\pmod{4}$ and $q_i\equiv 3\pmod{4}$, 
\begin{equation}\label{eqn:key_1}
a(m)=4(1+a_1)\cdots (1+a_s). 
\end{equation}
For $m=p_1^{a_1} \cdots p_s^{a_s}q_{s+1}^{a_{s+1}}\cdots q_u^{a_u}$, 
where $p_i\equiv 1\pmod{4}$ and $q_i\equiv 3\pmod{4}$ 
we define the sets in (\ref{eqns:sets}). 
Because of the equation (\ref{eqn:key_1}), the following equation holds: 
\[
\frac{\vert W((\ZZ^{2})_{m})\vert}{4}
=\frac{\vert W((\ZZ^{2})_{m})_{p_1}\vert}{4} \cdots \frac{\vert W((\ZZ^{2})_{m})_{p_s}\vert}{4}. 
\]
Then, as we obtained equation (\ref{eqn:design}), 
we obtain the following equation: 
\begin{eqnarray*}
I_{(\ZZ^{2})_{m}}(t)=\prod_{k=1}^{s}I_{(\ZZ^{2})_{m},p_k}(t). 
\end{eqnarray*}
Let $\varphi_k\in (0, 2\pi)$ be the minimum argument 
$z\in\OOO_{\QQ(\sqrt{-1})}$ whose norm is $p_k$. 
Then, 
\[
I_{(\ZZ^{2})_{m},p_k}(4)=\frac{\sin(4(1+a_k)\varphi_k)}{(1+a_k)\sin (4\varphi_k)} 
\]
since equation (\ref{eqn:cos}) holds. 
Let $\alpha$ be the least value of $a_k$
for which $I_{(\ZZ^{2})_{m},p_k}(4)=0$. 
If we assume that $\alpha >1$ then 
\[
\frac{\sin(4(1+\alpha)\varphi_{k})}{(1+\alpha)\sin (4\varphi_{k})}=0, 
\]
that is, $4\varphi_k =n\pi /(1+\alpha)$ for some $n\in \ZZ$. 
On the other hand, for $a_k=1$ 
\begin{eqnarray}
\frac{\sin(8\varphi_k)}{2\sin (4\varphi_k)}&=&\cos (4\varphi_k) \nonumber \\
&=&8\cos^4 (\varphi_k) - 8\cos^2 (\varphi_k) +1. \label{eqn:cos_2}
\end{eqnarray}
Here, we set $z:=2\cos (4\varphi_k)$. 
The number $z$ being twice the cosine of a rational multiple of $2 \pi$, 
is an algebraic integer. 
Moreover, if we set $e^{i(\varphi_k)}:=a+b\sqrt{-1}$ then 
$\cos (\varphi_k)=a/\sqrt{p_k}$ and 
because of the equation (\ref{eqn:cos_2}), $z$ is a rational number, 
namely, a rational integer. 
Therefore, $z=\pm 1$ or $z=\pm 2$. 
If $z=\pm 1$ then $\varphi_k=\pi/12$, $2\pi/12$, $4\pi/12$ or $5\pi/12$. 
However, $\tan \varphi_k$ is an irrational number 
and $b/a$ is a rational number. This is a contradiction. 
If $z=\pm 2$ then $\varphi_k=0$ or $\pi/4$, 
that is, $|X((\ZZ^{2})_m)_{p_k}|=4$. 
However, $|X((\ZZ^{2})_m)_{p_k}|=8$ since $p_k\equiv 1\pmod{4}$. 
This is a contradiction. 

Hence, it is enough to show that 
when $\alpha=1$, $I_{(\ZZ^{2})_m,p_k}(4)\neq 0$. 
If 
\[
I_{(\ZZ^{2})_m,p_k}(4)=\cos (4\varphi_k) =0 
\]
then $\varphi_k =\pi/8$ or $3\pi/8$. 
However, $\tan \pi/8$ and $\tan 3\pi/8$ are irrational numbers 
and $b/a$ is a rational number. 
This is a contradiction. 
Therefore $(\ZZ^{2})_m$ is not a spherical $4$-design. 
For $m^{\prime}=2^cm$, 
$W((\ZZ^{2})_{m^{\prime}})$ is rotated $k\pi$ for some $k\in\{0,1,2,3\}$ 
from $W((\ZZ^{2})_m)$. 
Therefore $(\ZZ^{2})_{m^{\prime}}$ is not a spherical $4$-design. 

\subsection{Calcut's results and the multiplicative property}\label{subsection:mul}
In Section \ref{section:non_Z^2_1} and \ref{section:non_Z^2_2}, 
we showed that the case of the $\ZZ^{2}$-lattice using Calcut's result and 
the multiplicative property of the 
Fourier coefficients of the theta series associated with 
the $\ZZ^{2}$-lattice respectively. 
In this section, 
we show that Culcut's result is essentially equivalent to 
the multiplicative property (\ref{eqn:mul_Z^2_1}) of the theta series: 
\[
\Theta_{\ZZ^{2}}(q)=\sum_{m=0}^{\infty} \vert W((\ZZ^{2})_{m})\vert q^{m}. 
\]
In Lemma \ref{lem:mul_prime_power} and Section \ref{section:non_Z^2_1}, 
we showed that the multiplicative property (\ref{eqn:mul_Z^2_1}) 
using Calcut's result. 

On the other hand, we assume that the multiplicative property, 
namely, equation (\ref{eqn:mul_Z^2_1}). 
Let $z\in\ZZ[\sqrt{-1}]$ be a Gaussian integer 
such that $\arg z\not\in \{0, \pm\pi/4, \pm\pi/2, \pm 3\pi/4, \pi\}$ 
and 
let $N((z))=p_1^{a_1}\cdots p_s^{a_s}
q_{s+1}^{a_{s+1}}\cdots q_u^{a_u}2^{a_{u+1}}$, 
where $(d_K/p_i)=1$ and $(d_K/q_i)=-1$ and 
let $(z)=P_1^{a_1}\cdots P_s^{a_s}
Q_{s+1}^{a_{s+1}}\cdots Q_u^{a_u}(1+\sqrt{-1})^{a_{u+1}}$
be the prime ideal factorization, 
where $N(P_i)=p_i$ and $N(Q_i)=q_i^2$. 
Then, $(z)^n$ 
and 
$\overline{(z)^n}$ 
are ideals of norm $N(z)^n$ 
because $P_i$ and $\overline{P_i}$ 
are prime ideals of norm $p_i$. 
We assume that $z^n\in\RR$. 
Then, we have $(z)^n=\overline{(z)^n}$. 
Therefore, the number of the nonequivalent ideals of 
norm $N(z)^{n}$ is less than $(1+a_1)\cdots(1+a_s)$. 
This is a contradiction since because of the multiplicative property, 
the number of the nonequivalent ideals of 
norm $N(z)^{n}$ is $(1+a_1)\cdots(1+a_s)$. 
Hence, the multiplicative property 
is equivalent to Calcut's result, namely, Lemma \ref{lem:Calcut_Z^2}. 

\subsection{The general cases whose class number is $1$}
In this subsection, 
we prove Theorem \ref{thm:A_2} and \ref{thm:gen} 
without using Calcut's result. 
\subsubsection{The case of $A_2$-lattice}\label{section:non_A_2}
We denote by $a(m)$ the $m$-th coefficient of the theta series of 
\[
\Theta_{A_{2}}(q)=\sum_{i=0}^{\infty}a(i)q^i . 
\]
and set $a^{\prime}(m):=a(m)/6$. 
Then, the function $a^{\prime}(m)$ is multiplicative and 
when $m$ is a prime power, $a^{\prime}(p^e)=F(p^e)$. 
Therefore, 
for $m=p_1^{a_1} \cdots p_s^{a_s}$$q_{s+1}^{a_{s+1}}$
$\cdots$$q_u^{a_u}3^{c}$, 
where $p_i\equiv 1\pmod{3}$ and $q_i\equiv 2\pmod{3}$, 
\begin{equation}\label{eqn:key_A_2}
a(m)=6(1+a_1)\cdots (1+a_s). 
\end{equation}

For $m=p_1^{a_1} \cdots p_s^{a_s}q_{s+1}^{a_{s+1}}\cdots q_u^{a_u}$, 
where $p_i\equiv 1\pmod{3}$ and $q_i\equiv 2\pmod{3}$ 
we define the sets in (\ref{eqns:sets}). 
Let $I_{{(A_{2})}_{m}}(t)$ and $I_{{(A_{2})}_{m}, p_k}(t)$ be the functions defined by (\ref{func:sets}). 
Because of the equation (\ref{eqn:key_A_2}), the following equation holds: 
\[
\frac{\vert W((A_2)_{m})\vert}{6}
=\frac{\vert W((A_2)_{m})_{p_1}\vert}{6} \cdots 
\frac{\vert W((A_2)_{m})_{p_s}\vert}{6}. 
\]
Then, as we obtained the equation (\ref{eqn:design}), 
we obtain the following equation: 
\begin{eqnarray*}
I_{(A_2)_m}(t)=\prod_{k=1}^{s}I_{(A_2)_m,p_k}(t). 
\end{eqnarray*}
Let $\varphi_k \in (0,2\pi)$ be the minimum argument 
$z\in\OOO_{\QQ(\sqrt{-3})}$ whose norm is $p_k$. 
Then, 
\begin{equation}\label{eqn:Lehmer_A_2}
I_{(A_2)_m,p_k}(6)=\frac{\sin(6(1+a_k)\varphi_k)}{(1+a_k)\sin (6\varphi_k)} 
\end{equation}
since the equation (\ref{eqn:cos}) holds. 
Let $\alpha$ be the least value of $a_k$
for which $I_{(A_2)_m,p_k}(6)=0$. 
If we assume that $\alpha >1$ then 
\[
\frac{\sin(6(1+\alpha)\varphi_{k})}{(1+\alpha)\sin (6\varphi_{k})}=0, 
\]
that is, $6\varphi_k =n\pi /(1+\alpha)$ for some $n\in \ZZ$. 
On the other hand, for $a_k=1$ 
\begin{eqnarray}
\frac{\sin(12\varphi_k)}{2\sin (6\varphi_k)}&=&\cos (6\varphi_k) \nonumber \\
&=&32\cos^{6}(\varphi_k)-48\cos^{4}(\varphi_k)+18\cos^{2}(\varphi_k)-1 . 
\label{eqn:cos_A_2}
\end{eqnarray}
Here, we set $z:=2\cos (6\varphi_k)$. 
The number $z$ being twice the cosine of a rational multiple of $2 \pi$, 
is an algebraic integer. 
Moreover, if we set $e^{i(\varphi_k)}=a+b\,\theta_d$, 
where $\theta_d$ is define in (\ref{df:theta_d}) 
then 
$\cos \varphi_k=(a+(b/2))/\sqrt{p_k}$ and 
because of the equation (\ref{eqn:cos_A_2}), $z$ is a rational number, 
namely, a rational integer. 
Therefore, $z=\pm 1$ or $z=\pm 2$. 
If $z=\pm 1$ then $\varphi_k=\pi/18$, $2\pi/18$, $4\pi/18$ or $5\pi/18$. 
However, if $\varphi_k=\pi/18$ or $5\pi/18$ then 
\begin{eqnarray*}
\frac{1}{2}=\sin(3\varphi_k)&=&3\sin(\varphi_k)-4\sin^{3}(\varphi_k)\\
\displaystyle&=&\frac{3\sqrt{3}b(p_k-b^2)}{2p_k\sqrt{p_k}}. 
\end{eqnarray*}
This is a contradiction. 
If $\varphi_k=2\pi/18$ or $4\pi/18$ then 
\begin{eqnarray*}
\pm\frac{1}{2}=\cos(3\varphi_k)&=&4\cos^{3}(\varphi_k)-3\cos(\varphi_k)\\
\displaystyle&=&\frac{((2a+b)^2-3p_k)(2a+b)}{2p_k\sqrt{p_k}}. 
\end{eqnarray*}
This is a contradiction. 
If $z=\pm 2$ then $\varphi_k=0$ or $\pi/6$. 
If $\varphi_k=0$ then $|X((A_2)_m)_{p_k}|=6$. 
However, $|X((A_2)_m)_{p_k}|=12$ since $p_k\equiv 1\pmod{3}$. 
This is a contradiction. 
If $\varphi_k=\pi/6$ then $\sin \varphi_k=1/2$, 
that is, a rational number and 
$\sqrt{3}b/(2\sqrt{p_k})$ is an irrational number 
since $p_k\equiv 1\pmod{3}$. 
This is a contradiction. 

Hence, it is enough to show that when $\alpha=1$, $I_{(A_2)_m,p_k}(6)\neq 0$. 
If 
\[
I_{(A_2)_m,p_k}(6)=\cos (6\varphi_k) =0 
\]
then $\varphi_k =\pi/12$ or $3\pi/12$. 
If $\varphi_k =\pi/12$ then 
\[
\frac{1}{2}=\sin (2(\pi/12))=2\sin (\pi/12)\cos (\pi/12)
=\frac{\sqrt{3}b(2a+b)}{2p_k}. 
\]
If $\varphi_k =3\pi/12$ then 
\[
1=\sin (2(3\pi/12))=2\sin (3\pi/12)\cos (3\pi/12)
=\frac{\sqrt{3}b(2a+b)}{2p_k}. 
\]
These are contradictions since $p_k\equiv 1\pmod{3}$. 
Therefore $(A_2)_m$ is not a spherical $6$-design. 
For $m^{\prime}=3^cm$, 
$W((A_2)_{m^{\prime}})$ is rotated 
$k\pi/3$ for some $k\in \{0,1,\ldots,5\}$ from $W((A_2)_{m})$. 
Therefore $(A_2)_{m^{\prime}}$ is not a spherical $6$-design. 

\subsubsection{The general cases whose class number is $1$}\label{section:non_gen}
Let $L$ be the lattices whose class number is $1$ 
except for the cases $\ZZ^{2}$- and $A_2$-lattice. 
We denote by $a(m)$ the $m$-th coefficient of the theta series of 
\[
\Theta_{L}(q)=\sum_{i=0}^{\infty}a(i)q^i . 
\]
and set $a^{\prime}(m):=a(m)/2$. 
Then, the function $a^{\prime}(m)$ is multiplicative and 
when $m$ is a prime power, $a^{\prime}(p^e)=F(p^e)$. 
Therefore, 
for $m=p_1^{a_1} \cdots p_s^{a_s}$$q_{s+1}^{a_{s+1}}\cdots q_u^{a_u}$
$r_{u+1}\cdots r_{v}^{a_{v}}$, 
where $(d_{K}/p_i)= 1$, $(d_K/q_i)= -1$, and $r_i\mid d_K$, 
\begin{equation}\label{eqn:key_gen}
a(m)=2(1+a_1)\cdots (1+a_s). 
\end{equation}

For $m=p_1^{a_1} \cdots p_s^{a_s}q_{s+1}^{a_{s+1}}\cdots q_u^{a_u}$, 
we define the sets in (\ref{eqns:sets}). 
Let $I_{L_m}(t)$ and $I_{L_m,p_k}(t)$ be the functions 
defined by (\ref{func:sets}). 
Because of the equation (\ref{eqn:key_gen}), the following equation holds: 
\[
\frac{\vert W(L_{m})\vert}{2}
=\frac{\vert W(L_{m})_{p_1}\vert}{2} \cdots 
\frac{\vert W(L_{m})_{p_s}\vert}{2}. 
\]
Then, as we obtained the equation (\ref{eqn:design}), 
we obtain the following equation: 
\begin{eqnarray*}
I_{L_m}(t)=\prod_{k=1}^{s}I_{L_m,p_k}(t). 
\end{eqnarray*}
Let $\varphi_k \in (0,2\pi)$ be the minimum argument 
$z\in\OOO_{\QQ(\sqrt{-d})}$ whose norm is $p_k$. 
Then, 
\begin{equation}\label{eqn:Lehmer_gen}
I_{L_m,p_k}(2)=\frac{\sin2(1+a_k)\varphi_k}{(1+a_k)\sin 2\varphi_k} 
\end{equation}
since the equation (\ref{eqn:cos}) holds. 
Let $\alpha$ be the least value of $a_k$
for which $I_{L_m,p_k}(2)=0$. 
If we assume that $\alpha >1$ then 
\[
\frac{\sin(2(1+\alpha)\varphi_{k})}{(1+\alpha)\sin (2\varphi_{k})}=0, 
\]
that is, $2\varphi_k =n\pi /(1+\alpha)$ for some $n\in \ZZ$. 
On the other hand, for $a_k=1$ 
\begin{eqnarray}
\frac{\sin(4\varphi_k)}{2\sin (2\varphi_k)}&=&\cos (2\varphi_k) \nonumber \\
&=&2\cos^{2} (\varphi_k)-1 . \label{eqn:cos_gen}
\end{eqnarray}
Here, we set $z:=2\cos (2\varphi_k)$. 
The number $z$ being twice the cosine of a rational multiple of $2 \pi$, 
is an algebraic integer. 
Moreover, if we set $e^{i(\varphi_k)}=a+b\,\theta_d$ then 
$\cos (\varphi_k)=(\mbox{Re}\,(a+b\,\theta_d))/\sqrt{p_k}$ and 
because of the equation (\ref{eqn:cos_gen}), $z$ is a rational number, 
namely, a rational integer. 
Therefore, $z=\pm 1$ or $z=\pm 2$. 
If $z=\pm 1$ then $\varphi_k=\pi/6$, $2\pi/6$, $4\pi/6$ or $5\pi/6$ and 
$\tan \pi/6=1/\sqrt{3}$, $\tan 2\pi/6=\sqrt{3}$, 
$\tan 4\pi/6=-\sqrt{3}$ or $\tan 5\pi/6=-1/\sqrt{3}$. 
However, 
\[
\tan \varphi_k =
\left\{
\begin{array}{ll}
\displaystyle\frac{\sqrt{2}b}{a} &\mbox{ if}\ d=2 \\
\displaystyle\frac{b\sqrt{d}}{2a+b} &\mbox{ otherwise}. 
\end{array}
\right.
\]
This is a contradiction. 
If $z=\pm 2$ then $\varphi_k=0$ or $\pi/2$, 
that is, $a=0$ or $b=0$. 
This is a contradiction since 
$a^2+b^2=p_k$ and $p_k$ is a prime number. 

Hence, it is enough to show that when $\alpha=1$, $I_{L_m,p_k}(2)\neq 0$. 
If 
\[
I_{L_m,p_k}(2)=\cos (2\varphi_k) =0 
\]
then $\varphi_k =\pi/4$ or $3\pi/4$. 
However, it is impossible because 
$\mbox{Re}\,(a+b\,\theta_d)\neq \pm\mbox{Im}\,(a+b\,\theta_d)$. 
Therefore $L_m$ is not a spherical $2$-design. 
For $m^{\prime}=r_{u+1}^{a_{u+1}}\cdots r_v^{a_v}m$, 
where $r_i\mid d_K$, 
$W(L_{m^{\prime}})$ is rotated $k\pi$ for some $k\in\{0,1\}$ from $W(L_{m})$. 
Therefore $L_m$ is not a spherical $2$-design. 

\begin{rem}
We remark that equations (\ref{eqn:Lehmer_Z^2}), 
(\ref{eqn:Lehmer_A_2}) and (\ref{eqn:Lehmer_gen}) 
are essentially same as the equation (2) which appeared in 
page $3$ of \cite{Toy-BM} and 
the equation (5) which appeared in page $5$ of \cite{Toy-BM2}. 
However, the ways to obtain the first three 
equations (\ref{eqn:Lehmer_Z^2}), 
(\ref{eqn:Lehmer_A_2}) and (\ref{eqn:Lehmer_gen}) and 
the others are different from each other. 
After we calculate the right hand side 
of the definition (\ref{func:sets}), 
we obtained the first three equations. 
On the other hand, 
using the recurrence relation of the coefficients of the 
weighted theta series associated with the 
lattice, which is the property of the normalized Hecke eigenform, 
we obtained the others. 
\end{rem}
\section{Generalization of Calcut's results}\label{subsection:calcut}
In section \ref{subsection:calcut}, 
we quote and generalize Calcut's results. 
Let $K$ be an imaginary quadratic field whose class number is $1$. 
\begin{thm}\label{thm:Calcut}
Let $z\neq 0$ be an element of $\OOO_{K}$. 
There is a natural number $n$ such that 
$z^{n}$ is a real number if and only if 
$\arg z$ is a multiple of 
\[
\left\{
\begin{array}{ll}
\pi/4 \mbox{\ \ if } K=\QQ(\sqrt{-1})\\
\pi/6 \mbox{\ \ if } K=\QQ(\sqrt{-3})\\
\pi/2 \mbox{\ \ otherwise}. 
\end{array} 
\right.
\]
\end{thm}
\begin{proof}
In \cite{Calcut}, 
Calcut show the case $K=\QQ(\sqrt{-1})$. 
Therefore, we show the case $K=\QQ(\sqrt{-3})$ and 
the others can be proved similarly. 

If $\arg z$ is a multiple of $\pi/6$ then 
$z^{6}$ is a real number. 
Assume that $z^{n}=m$, where $z=a+b\theta_d$, 
$\theta_d =(-1+\sqrt{-3})/2$ and $m\in\RR$. 
It is enough to show that $z$ is a nonunit and primitive, 
that is, $\gcd(a,b)=1$. 
Let $(z)=P_1^{a_1}\cdots Q_1^{b_1}\cdots ((3+\sqrt{-3})/2)^{c}$, 
where $(d_K/N(P_i))=1$, 
$N(Q_i)=q_i^2$ and $(d_K/q_i)=-1$, 
be the prime ideal factorization. 
Since $z$ is a primitive, 
we have $b_i=0$. 
Moreover, the condition $z^{n}=m$ implies 
$\overline{P_i}\mid (z)$, that is, 
$P_i\overline{P_i}\mid (z)$. 
Therefore, $a_i=0$ since $z$ is a primitive. 
So, the proof is completed. 
\end{proof}
\begin{cor}
Let $\ZZ+\ZZ\,\theta_d$ be the integer ring of 
an imaginary quadratic field whose class number is $1$. 
\begin{enumerate}
\item 
If $\tan (k\pi/n)=({\rm Im}\,(z))/({\rm Re}\,(z))=b/a$ 
for some $z\in \ZZ+\,\ZZ\sqrt{-1}$ 
then $\tan (k\pi/n)=0$ or $\pm 1$. 

\item 
If $\tan (k\pi/n)=({\rm Im}\,(z))/({\rm Re}\,(z))=\sqrt{3}b/(2a+b)$ 
for some $z\in \ZZ+\ZZ\,(-1+\sqrt{-3})/2$ 
then $\tan (k\pi/n)=0$, $\pm 1/\sqrt{3}$ or $\pm \sqrt{3}$. 

\item 
If $\tan (k\pi/n)=({\rm Im}\,(z))/({\rm Re}\,(z))$ 
for some $z\in \ZZ+\ZZ\,\theta_d$ 
then $\tan (k\pi/n)=0$. 
\end{enumerate}
\end{cor}
\begin{proof}
We remark $z^n\in \RR$. 
Then, using Theorem \ref{thm:Calcut} 
the proof is completed. 
\end{proof}

\begin{cor}\label{cor:Calcut}
Let the notation be the same as above. 
Let $z_i=a+b\theta_d \in \ZZ+\ZZ\,\theta_d$. 
If 
\[
\displaystyle\frac{k\pi}{n}
=\sum_{j=1}^{l}m_j\arctan\frac{{\rm Im}\,(z_j)}{{\rm Re}\,(z_j)}
\]
holds, then 
\[
\frac{k\pi}{n}=\left\{
\begin{array}{ll}
j\pi/4 \mbox{\ \ if } K=\QQ(\sqrt{-1})\\
j\pi/6 \mbox{\ \ if } K=\QQ(\sqrt{-3})\\
j\pi/2 \mbox{\ \ otherwise} 
\end{array} 
\right.
\]
for some $j\in\ZZ$. 
\end{cor}
\begin{proof}
We have 
\begin{eqnarray*}
\frac{k\pi}{n}
&=&\sum_{j=1}^{l}m_j\arctan\frac{{\rm Im}\,(z_j)}{{\rm Re}\,(z_j)} \\ 
&=&\sum_{j=1}^{l}m_j\arg(a_j+b_j\theta_d) \\ 
&=&\arg\prod_{j=1}^{l}(a_j+b_j\theta_d)^{m_j} \pmod{2\pi}. 
\end{eqnarray*}
We set 
\[
z:=\prod_{j=1}^{l}(a_j+b_j\theta_d)^{m_j} . 
\]
We remark $z^n\in \RR$. 
Then, using Theorem \ref{thm:Calcut} 
the proof is completed. 
\end{proof}


\bigskip
\noindent
{\bf Acknowledgment.}
The authors would like to thank S.\ V.\ Konyagin 
for the useful discussions. 
He has an interesting question that 
is there a circle satisfying the following condition that 
the $\ZZ^{2}$-lattice points on the circle make a $4$-design?
The second author is supported by JSPS research fellowship. 
The third author was supported by the Russian Foundation 
for Basic Research (project 08-01-00501). 

\end{document}